\documentclass[11pt]{amsart}
\usepackage{amsfonts,amsmath,amssymb,amsthm}
\usepackage{enumitem,comment,multicol,ulem}
\usepackage{hyperref}
\usepackage[textsize=tiny]{todonotes}
\usepackage[dvipsnames]{xcolor}
\usepackage{graphicx} 
\usepackage{thm-restate}
\usepackage{pinlabel}
\usepackage{float}

\newcommand\arf{\operatorname{arf}}

\newtheorem{theorem}{Theorem}[section]
\newtheorem{corollary}[theorem]{Corollary}

\newtheorem{proposition}[theorem]{Proposition}

\theoremstyle{definition}
\newtheorem{definition}[theorem]{Definition}
\newtheorem{remark}[theorem]{Remark}

\newtheorem{question}[theorem]{Question}
\newtheorem*{question*}{Question}

\newtheorem{notation}[theorem]{Notation}
\newtheorem{example}[theorem]{Example}

\title[Ising and Potts invariants]{Three formulas for the Ising and Potts invariants of a knot or link}
\author{Tristan Bullock$^1$}
\address{$^1$%
Medical University of South Carolina, Department of Public Health Sciences, 135 Cannon Street,
Suite 303 MSC 835,
Charleston, SC 29425
}
\email{bullockt@musc.edu}

\author{Thomas Kindred$^2$}
\address{$^2$%
Burton Hall 115, Smith College,
Northampton, MA 01063}
\email{tkindred@smith.edu}


\begin{document}

\begin{abstract}
    Jones described how the Ising and Potts models from statistical mechanics give rise, with appropriate choices of Boltzmann weights, to invariants of an oriented link. The Boltzmann weights that Jones proposed, however, work only with a correction factor that he does not mention.  We fill in the missing details in two different ways.  We also show that all of these invariants, which we call the Ising and Potts invariants, are given by evaluating the Jones polynomial at $t$ equal to one of these Boltzmann weights.
\end{abstract}

\maketitle

\section{Introduction}

In the early 1920's, Lenz developed a statistical mechanical model that, he hoped, might provide a mechanism by which individual particle interactions could give rise, in the aggregate, to abrupt phase changes at certain critical temperatures. The model came to be known as the Ising model, after Lenz's student Ising, solved the 1-dimensional version of the model in 1925.  In 1944, Onsager solved the 2-dimensional square lattice version of the Ising model and found that, unlike the 1-dimensional version, it captured phase change. In 1982, building on a technical innovation within Onsager's work known as the {\it triangle-star} relation, Baxter provided a new, more robust solution to the same model, which was also well-suited to a more general set of models named after Potts.

Shortly after the discovery of the Jones polynomial in 1985, Kauffman and Jones recognized a connection between the new polynomial and the spin models of Ising and Potts. The main idea is that any knot or link diagram $D$ on $S^2$ (with a choice of checkerboard shading) determines a {\it Tait graph} $G$ (see \textsection\ref{S:Back} for details), on which one can impose an Ising or Potts model---an added complication, however, is that each edge in a Tait graph carries a sign, and the edge energies in the imposed spin model might depend on these edge-signs as well as the labels of the incident vertices. There are thus four edge energies, $E^+_=$, $E^+_{\neq}$, $E^-_=$, and $E^-_{\neq}$.  In computations, it is more convenient to work with the {\it Boltzmann weights} $b^{\pm}_==\exp(-E^\pm_=/kT)$ and $b^{\pm}_{\neq}=\exp(-E^\pm_{\neq}/kT)$, where $k$ is Boltzmann's constant and $T$ is temperature in degrees Kelvin.

This setup gives rise to a natural question: is it possible to assign edge energies (or, equivalently, Boltzmann weights) such that the resulting partition function is an invariant of the original knot or link?  Intriguingly, the condition for the partition function to be invariant under Reidemeister-3 moves turns out to be (an edge-signed version of) the same triangle-star relation that Onsager discovered and Baxter exploited as a computational tool for the physical models. In 1990, Jones proved that it is possible to assign edge energies in such an invariant fashion. Each spin model thus gives rise to an invariant of (oriented) knots and links. In the case of the Ising model, we call this invariant the {\it Ising invariant} of a knot or link $K$, and we denote it $\text{Ising}(K)$.  Likewise, for each integer $n\geq 3$, there is an associated Potts model, and an associated invariant that we call the {\it $n^{\text{th}}$ Potts invariant} of $K$, denoting it $\text{Potts}_n(K)$.

Jones captures the story and its implications nicely in two articles.  One, in Scientific American, presents the basic ideas of these surprising connections to a lay audience, while suppressing plenty of details \cite{jones_SA}.  The other, in the Pacific Journal, delves deeply into spin models, their ramifications, and generalizations \cite{jones_pacific}. Here, Jones defines spin models on edge-signed graphs via a system of equations, and then he describes how to take any solution to the system and build it into an invariant of oriented links.

In Example 2.17 in that paper, as Jones briefly narrows his focus from spin models in general to the Ising and Potts models specifically (with $n=2$ and $n\geq 3$, respectively), he suggests using the Boltzmann weights $b^+_{=}=1=b^-_{=}$, $b^+_{\neq}=t$, and $b^-_{\neq}=t^{-1}$, where $t+t^{-1}+2=n$.  The trouble is that these are not actually solutions to Jones' systems of equations.\footnote{This error is repeated in Exercises 7.10 and 7.11 of Adams' {\it The Knot Book} \cite{knotbook}.} In our main result, we determine the correct choices of Boltzmann weights based on Jones' system of equations, and then we also show that Jones' proposed values for these weights, which we typically denote with $B$'s rather than $b$'s, can also be reconciled by means of a correction factor:
\begin{restatable}{theorem}{maintheorem}\label{T:Main}
Let $D$ be a diagram of an oriented link $K\subset S^3$, and let $w(D)$ denote its writhe. Also let $G$ be a Tait graph of $D$, with $V$, $E_+$, and $E_-$ respectively denoting the numbers of vertices, positive edges, and negative edges in $G$.
    Then, for any integer $n\geq 3$, writing 
    \begin{equation}\label{E:x}\tag{$x$}
x=\sqrt{n}\left(\sqrt{n}(3-n)+ (n-1)\sqrt{n-4}\right)^{-1/2},
    \end{equation}
    \begin{equation}\label{E:y}\tag{$y$}
        y=\frac{x}{2}\left(2-n+\sqrt{n(n-4)}\right),
    \end{equation}
    and 
    \begin{equation}\label{E:t}\tag{$t$}
        t=\frac{y}{x}=\frac{1}{2}\left(2-n+\sqrt{n(n-4)}\right),
    \end{equation}
    and using as Boltzmann weights either
    \[b_+^{=}=x,~b_-^{=}=x^{-1},~b_+^{\neq}=y,\text{ and }b_-^{\neq}=y^{-1},\]
    or 
    \[B_+^{=}=1=B_-^{=},~B_+^{\neq}=t^{-1},\text{ and }B_-^{\neq}=t,\]
    the $n^{\text{th}}$ Potts invariant of $K$ is given by either of the following equivalent formulas:
    \begin{align}\label{E:Potts}
         \text{Potts}_n(K)=&
         \left(\frac{1}{\sqrt{n}}\right)^{V+1}  x^{w(D)+E_--E_+}
         \sum_{\text{states}}\prod_{\text{edges }e}B(e)
         \\
         \label{E:Potts_with_CF}
    \text{Potts}_n(K)=&
    \left(\frac{1}{\sqrt{n}}\right)^{V+1}  x^{-w(D)}
    \sum_{\text{states}}\prod_{\text{edges }e}b(e).
    \end{align}
\end{restatable}
Equations \eqref{E:Potts} and \eqref{E:Potts_with_CF} are two of the three formulas advertised in the title of this paper. It is interesting to note in Equation \eqref{E:Potts} that $w(D)+E_--E_+$ is the net boundary slope of a checkerboard surface from $D$, namely the one that does not have $G$ as a spine. That is, if $F$ and $W$ are the checkerboard surfaces from $D$, where one can view $G$ as a spine of $F$, then $w(D)+E_--E_+$ is the linking number of $K$ with a co-oriented pushoff of $K$ in $W$.

We also reiterate and extend another point made by Jones by showing that the Ising and Potts invariants of a knot or oriented link equal certain evaluations of the Jones polynomial (see the last sentence of Example \ref{Ex:2.17} for Jones' description):
\begin{restatable}{theorem}{jonestheorem}\label{T:Jones}
With the same setup as Theorem \ref{T:Main}, we have 
\begin{equation}\label{E:third}
\text{Potts}_n(K)=V_K(t)
\end{equation}
for each $n\geq 2$. In particular, 
\[\text{Ising}(K)=V_K(i)=\begin{cases}(-\sqrt{2})^{|K|-1}(-1)^{\arf(K)}&\text{if }\arf(K)\text{ is defined}\\
0&\text{otherwise}.
\end{cases}\]
\end{restatable}
Equation \eqref{E:third} is the third formula advertised in the title.  While all of these formulas can admittedly be obtained through a careful and skeptical reading of the first few pages of \cite{jones_pacific}, our experience was that the effort and reward of this task were together sufficient to merit writing them up for this paper.  The reward is a clarified narrative about these connections between statistical mechanics and the Jones polynomial, in which spanning surfaces and even the arf invariant make notable cameo appearances. The effort, though elementary, is notable in light of the broad usage that \cite{knotbook} receives in undergraduate instruction---we hope that these corrections and clarifications, however, minor and elementary, might help rejuvenate interest among future students and scholars in the fascinating connections between quantum knot invariants and statistical physics.

\section{Background}\label{S:Back}

\subsection{The Ising and Potts models}

In the early 1920's, Lenz developed a statistical mechanical model that, he hoped, might provide a mechanism by which individual particle interactions could give rise, in the aggregate, to abrupt phase changes at certain critical temperatures. 
In this model, a physical system is viewed as a graph, with vertices representing atoms, each in either of two possible states, and edges representing interactions between certain nearby pairs of atoms, with the energy $E_=$ or $E_{\neq}$ for each edge depending only on whether or not its atoms are in the same state.  A system with $n$ atoms thus has $2^n$ possible states.  For any particular state $S_0$, its energy $E(S_0)$ is the sum of the energies along its edges, and writing $Q(S_0)=\exp(-E(S_0)/kT)$, where $k$ is Boltzmann's constant and $T$ is temperature in degrees Kelvin, the probability of the system being in this state $S_0$ is given by $Q(S_0)/{\sum_{
S}Q(S)}$. The denominator is called the {\it partition function} of the system.

Lenz gave the problem of solving the 1-dimensional case of this model to his doctoral student, Ising, and Ising's solution demonstrated that, unfortunately, the model did not capture phase change, at least in the 1-dimensional case. 
Consequently, the model, which became known as the Ising (or Ising-Lenz) model, received relatively little attention in the ensuing years.  
In 1944, however, Onsager solved the 2-dimensional square lattice version of the Ising model and, intriguingly, found that it predicted sudden phase changes at critical temperatures.

One technical innovation within Onsager's work was the {\it triangle-star} relation, which held that any local replacement as shown in Figure \ref{Fi:Baxter} would multiply the model's partition function by a constant factor (depending only on the two edge energies $E_=$ and $E_{\neq}$ used throughout the system).  
In 1982, Baxter provided a new, more robust solution to the 2-dimensional square lattice Ising model, built entirely upon the triangle-star relation and its consequences, which was also well-suited to more general models in which each atom has not just two possible states, as in the Ising model, but $n$, where $n$ is any positive constant. Such a model with $n\geq 3$ is typically called a Potts model, provided each edge energy still depends only on whether or not the incident vertices are in the same state. (More general {\it spin models} are constructed the same way, but edge energies might depend on the actual states at their vertices, and not just on whether or not they match.)

\begin{figure}
    \centering
    \includegraphics[height=1in]{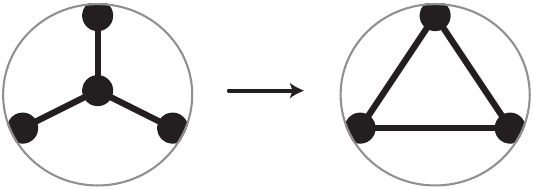}
    \caption{The replacement move leading to the triangle-star relation}
    \label{Fi:Baxter}
\end{figure}

\subsection{Knot theoretical preliminaries}

We assume throughout that all link diagrams are connected. Theorem 1 in \cite{coward} implies that, even for a split link $L$, all connected diagrams of $L$ are related by Reidemeister moves that preserve connectedness.

Any knot or link diagram $D\subset S^2$ cuts $S^2$ into disks, and it is possible to color these disks black and white in checkerboard fashion, so that regions of the same color abut only at crossings.  The {\it checkerboard surfaces} $F$ and $W$ from $D$ are formed from these black and white disks, respectively, by (deleting a neighborhood of each crossing and) attaching an appropriate half-twist band at each crossing, as shown in Figure \ref{Fi:CB_Band}.

\begin{figure}
    \centering
    \includegraphics[height=.5in]{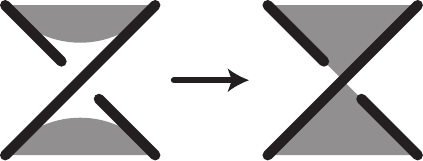}
    \hspace{1in}
    \includegraphics[height=.5in]{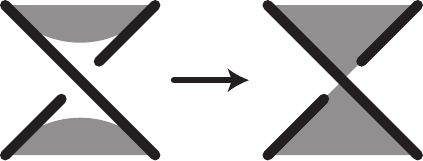}
    \caption{Construct a checkerboard surface by attaching shaded disks with a half-twist band at each crossing}
    \label{Fi:CB_Band}
\end{figure}

The {\it Tait graph} for $D$, say from its black checkerboard shading, is constructed by placing a vertex in each black region, inserting an edge through each crossing to join the black regions that are incident there, and labeling each edge with a sign as indicated in Figure \ref{Fi:Tait}.  Viewing this graph $G$ as a subset of the checkerboard surface $F$, one observes that $G$ is a spine of $F$. 

\begin{figure}
    \centering
    \labellist
    \pinlabel {{\color{red} $\boldsymbol{+}$}} at 185 40
    \pinlabel {{\color{red} $\boldsymbol{-}$}} at 555 40
    \endlabellist
    \includegraphics[height=.5in]{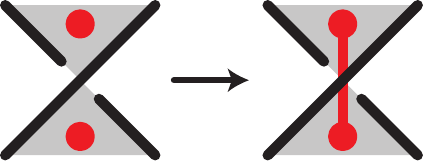}
    \hspace{1in}
    \includegraphics[height=.5in]{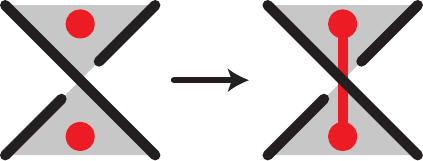}
    \caption{Construct a Tait graph by placing a vertex in each shaded region and inserting a signed edge through each crossing.}
    \label{Fi:Tait}
\end{figure}


\begin{definition}\label{D:slope}
    If $F$ is a spanning surface (orientable or nonorientable) for an oriented knot or link $K$, then the linking number of $K$ with a co-oriented pushoff of $K$ in $F$ is called the {\it net boundary slope} of $F$, denoted $\text{slope}(F)$
\end{definition}
In Definition \ref{D:slope}, if $K$ is a knot, then $\text{slope}(F)$ is simply called the {\it boundary slope} of $F$ and is independent of orientation. When $K$ is a link of multiple components, however, $\text{slope}(F)$ generally depends on the orientation of $K$ (the exception is if the pairwise linking numbers of the components of $K$ are all zero). Some readers may find it useful to note that $\text{slope}(F)=-e(F,K)$, where $e(F,K)=e(\Sigma)$ is the normal Euler number of the closed surface $\Sigma\subset S^4$ obtained by pushing the interior of $F$ into the interior of the 4-ball {\it below} the equatorial 3-sphere $S^3$, and then capping off the oriented link $\partial F=K$ with a (coherently oriented) Seifert surface in $S^3$.\footnote{The choice between the two 4-ball hemispheres matters because both inherit orientations from $S^4$, under which the induced boundary orientations on $S^3$ are opposite; it is customary to push into the lower hemisphere because this induces the standard orientation on $S^3$.  If instead one constructs a closed surface $\Sigma'\subset S^4$ by pushing the interior of $F$ into the interior of the upper hemisphere of $S^4$ and capping off with a Seifert surface in $S^3$, then $\text{slope}(F)=e(\Sigma)$.} For our purposes, the key fact about slopes of checkerboard surfaces is the following:

\begin{proposition}\label{P:Slopes}
    Let $D$ be a diagram of an oriented link $K\subset S^3$, and let $w(D)$ denote its writhe. Let $F$ and $W$ be the checkerboard surfaces from $D$, and let $G$ be the Tait graph of $D$ from $F$, with $E_+$ and $E_-$ denoting the numbers of positive and negative edges in $G$.  Then
    \begin{equation}\label{E:slopes}
    \begin{split}
    \text{slope}(F)=w(D)+E_+-E_-,\\\text{slope}(W)=w(D)+E_--E_+.
    \end{split}
    \end{equation}
\end{proposition}

\begin{proof}
    Figure \ref{Fi:Crossing_Slope} shows the contribution that each crossing makes to $\text{slope}(F)$ and $\text{slope}(W)$.  We thus confirm that
    \[\text{slope}(F)=2\#\raisebox{-6pt}{\includegraphics[height=18pt]{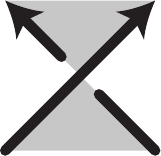}}-2\#\raisebox{-6pt}{\includegraphics[height=18pt]{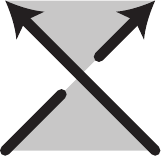}}=\#\raisebox{-6pt}{\includegraphics[height=18pt]{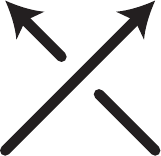}}-\#\raisebox{-6pt}{\includegraphics[height=18pt]{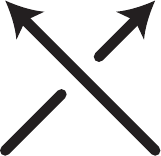}}+\#\raisebox{-6pt}{\includegraphics[height=18pt]{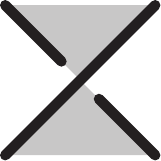}}-\#\raisebox{-6pt}{\includegraphics[height=18pt]{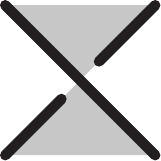}}=w(D)+E_+-E_-\]
    and
    \[\pushQED\qed\text{slope}(W)=2\#\raisebox{-6pt}{\includegraphics[height=18pt]{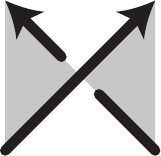}}-2\#\raisebox{-6pt}{\includegraphics[height=18pt]{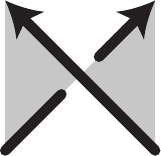}}=\#\raisebox{-6pt}{\includegraphics[height=18pt]{Figures/Crossing+}}-\#\raisebox{-6pt}{\includegraphics[height=18pt]{Figures/Crossing-}}+\#\raisebox{-6pt}{\includegraphics[height=18pt]{Figures/Crossing_-}}-\#\raisebox{-6pt}{\includegraphics[height=18pt]{Figures/Crossing_+}}=w(D)+E_--E_+.\qedhere\]
\end{proof}

\begin{figure}
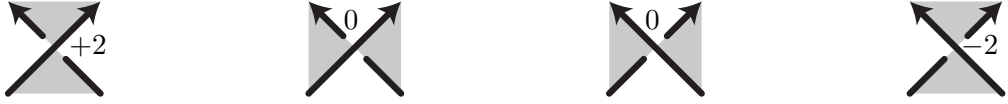

\centering
\labellist
    \pinlabel {$+2$} at 65 40
    \pinlabel {$0$} at 505 60
    \pinlabel {$0$} at 270 60
    \pinlabel {$-2$} at 760 40
    \endlabellist
    \includegraphics[height=.5in]{Figures/Crossing++}\hspace{1in}
    \includegraphics[height=.5in]{Figures/Crossing+-}\hspace{1in}
    \includegraphics[height=.5in]{Figures/Crossing-+}\hspace{1in}
    \includegraphics[height=.5in]{Figures/Crossing--}
    \caption{The contribution of each crossing to $\text{slope}(F)$
    \label{Fi:Crossing_Slope}}
\end{figure}

The last bit of knot theory we will use is the Jones polynomial. The following skein-theoretic definition will be the most convenient one for our purposes.

\begin{definition}\label{D:Jones}
    The Jones polynomial of an oriented link $K$, denoted $V_K(t)$, is the Laurent polynomial in $t^{1/2}$ that satisfies $V_\bigcirc(t)=1$ and the following {\it skein relation} \eqref{E:skein}, where $K_+,K_-,K_0\subset S^2$ are oriented link diagrams that are identical except in a disk where they differ as shown in Figure \ref{Fi:skein}.
    \begin{equation}\label{E:skein}
        t^{-1}V_{K_+}(t)-tV_{K_-}(t)=(t^{1/2}-t^{-1/2})V_{K_0}(t).
    \end{equation}
\end{definition}

\begin{figure}
    \centering
    \labellist
    \pinlabel {$K_+$} at 40 0
    \pinlabel {$K_-$} at 285 0
    \pinlabel {$K_0$} at 520 0
    \endlabellist
    \includegraphics[height=.5in]{Figures/Crossing+}
    \hspace{1in}
    \includegraphics[height=.5in]{Figures/Crossing-}
    \hspace{1in}
    \includegraphics[height=.5in]{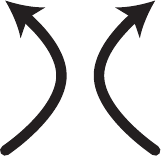}
    \caption{The local appearances of $K_+$, $K_-$, and $K_0$ in the 
    skein relation \eqref{E:skein}}
    \label{Fi:skein}
\end{figure}

\begin{remark}
In \cite{kauffman}, Kauffman describes how, in the case of alternating link diagrams, there are also interesting connections between the Potts models, the Jones polynomial, and the dichromatic polynomials of Tait graphs. Also see \cite{knotbook}.
\end{remark}

\subsection{Spin models}

The following definition of spin models is Jones', except that we denote the {\it Boltzmann weights} (described in the introduction) by $b_+$ and $b_-$, rather than Jones' $w_+$ and $w_-$, in order to reserve $w$ and $\omega$ respectively for the writhe $w(D)$ of an oriented knot or link diagram $D$ and the $16^{\text{th}}$ root of unity $\omega=e^{i\pi/8}$.


\begin{definition}[Definition 2.1 in \cite{jones_pacific}]
    A spin model $S = \{\theta, b_{\pm}\}$ is a finite set of spins $\theta$ together with functions $b_+,b_-:\theta\times \theta\to \mathbb{C}$ that satisfy the following conditions for all $a,r,c\in\theta$:
        \begin{equation}\tag{Jones' (2.2)}
            b_{\pm}(a,c) = b_{\pm}(c,a),
        \end{equation}
        
        \begin{equation}\tag{Jones' (2.3)}
            b_{+}(a, c)b_{-}(a, c) = 1,
        \end{equation}

        \begin{equation}\tag{Jones' (2.4)}
            \sum_{x \in \theta} b_{-}(a,x) b_{+}(x,c) = n \delta(a,c),
        \end{equation}
        
        \begin{equation}\tag{Jones' (2.5)}
            \sum_{x \in \theta} b_{+}(a,x) b_{+}(r,x) b_{-}(c,x) = \sqrt{n} b_{+}(a,r) b_{-}(r, c) b_{-}(c,a).
        \end{equation}
\end{definition}

A spin model in which each weight $b_{\pm}(a,c)$ depends only on (the subscript $\pm$ and) whether or not $a=c$ is called a {\it Potts model} for $n\geq 3$ or an {\it Ising model} for $n=2$.  With this added assumption, Jones' (2.2) becomes redundant.  

In the case of the Ising model, the functions $b_\pm$ take on at most four distinct values, $b_+^{=},b_+^{\neq},b_-^{=},b_-^{\neq}$:
\[b_{\pm}(a,c)=\begin{cases}b_{\pm}^{=}&a=c\\b_{\pm}^{\neq}&a\neq c.\end{cases}
\]

\begin{definition}[Definition 2.6 of \cite{jones_pacific}]\label{D:Partition}
    Let $D\subset \mathbb{R}^2$ be an unoriented diagram of a knot or link $L$. Shade the regions of $L$ black and white in checkerboard fashion such that the unbounded region is white, and let $G$ be the (edge-signed) Tait graph from the black regions. If $S=(\theta,b_{\pm})$ is a spin model, then a {\it state} of $G$ is an assignment $\sigma:V(G)\to\theta$ of a label at each vertex of $G$, and the {\it partition function} for $D$ and $S$ is
    \begin{equation}
    Z^S_D = \left(\frac{1}{\sqrt{n}}\right)^{\#V(G)+1} \sum_{\begin{smallmatrix}
        \text{states}\\
        \sigma:V(G)\to\theta
    \end{smallmatrix}} \prod_{\begin{smallmatrix}
        \text{edges}\\
        e=(u,v)
    \end{smallmatrix}} b_{\text{sign}(e)}(\sigma(u), \sigma(v))
\end{equation}
\end{definition}

\begin{remark}
    Jones' definition has a correction factor of $\left(\frac{1}{\sqrt{2}}\right)^{\#V(G)}$, rather than $\left(\frac{1}{\sqrt{2}}\right)^{\#V(G)+1}$.  We make this change to normalize so that $Z^S_\bigcirc=1$, and more generally to prepare Theorem \ref{T:Jones}.
\end{remark}

Jones proves:

\begin{theorem}[Theorem 2.8 and Proposition 2.13 of \cite{jones_pacific}]
Given a spin model $S=(\theta,b_{\pm})$, the value $Z_D^S$ is a regular invariant of unoriented links.  Moreover, taking any $c\in \theta$ and defining $A=b_-(c,c)$, the value $A^{w(D)}Z_D^S$ is an invariant of oriented links.
\end{theorem}

\subsection{An incomplete solution to the Ising and Potts models}

Jones then provides the following example, which seems to be in error, although we will see that it can be rectified with an appropriate correction factor. We quote Jones nearly verbatim here, other than that, as throughout, we change his $w_\pm$ notation to $b_\pm$.

\begin{example}[Example 2.17 of \cite{jones_pacific}]\label{Ex:2.17} (The [Ising and] Potts model[s]). In statistical mechanics the Potts model is the spin model for which the Boltzmann weights depend only on whether the two atoms are in the same state or not. Correct choice of the parameters leads, for each $n$, to the following choice of $b_\pm$
(by (2.3) it suffices to give $b_+$):
\begin{equation}\label{E:error}
b_+(a,c)=\begin{cases} 1&\text{if }a=c\\ -t^{-1}&\text{otherwise,}\\ \end{cases}
\end{equation}
where $2 + t + t^{-1}= n$. It is easy to check that these Boltzmann weights
satisfy (2.2)-(2.6) [sic]. One may also check directly that the invariant of
unoriented links is an unoriented version of $V_L(t)$.
\end{example}

To see that this is incorrect, note e.g. that for the Ising model, Jones proposes the Boltzmann weights $b_+^==1=b_-^=$, $b_+^{\neq}=\mp i$, and $b_-^{\neq}=\pm i$, but these values do not satisfy, e.g., Jones' (2.5) with $a=r=c$.  The same error is essentially repeated in \cite{knotbook}. 

Later, we will find that one can correct these proposed solutions by multiplying by an appropriate correction factor.  First, though, we directly solve Jones' (2.3)-(2.5) for the Potts and Ising models. 

\begin{notation}\label{N:IsingPotts}
    When $S=(\theta,b_\pm)$ is the Ising model, we denote Jones' invariant $(b_-^=)^{w(D)}Z^S_D(S)$ by $\text{Ising}(D)$.  When $S=(\theta,b_\pm)$ is the Potts or Ising model with $|\theta|=n$, we denote $(b_-^=)^{w(D)}Z^S_D(S)$ by $\text{Potts}_n(D)$.  Thus, $\text{Ising}(D)=\text{Potts}_2(D)$.
\end{notation}

\section{Main results}

Before proving our main results, we compute the Boltzmann weights for the Ising and Potts models directly from Jones' (2.3)-(2.5). Here, we adopt the same setup as Theorem \ref{T:Main}, with $n\geq 2$ arbitrary.  

\begin{proposition}\label{P:Main}
    With the same setup as Theorem \ref{T:Main}, with $n\geq 2$ arbitrary, the solutions to Jones' (2.3)-(2.5) are all given by 
\begin{equation}\label{E:Potts_general_x}
    b_+^==\pm\sqrt{2}\left(\sqrt{n}(3-n)+\varepsilon (n-1)\sqrt{n-4}\right)^{-1/2},\\
\end{equation}
\begin{equation}\label{E:Potts_general_y}
    b_+^{\neq}=\frac{1}{2}\left(2-n+\varepsilon \sqrt{n(n-4)}\right)b_+^=,
\end{equation}
$b_-^==(b_+^=)^{-1}$, and $b_-^{\neq}=(b_+^{\neq})^{-1}$, where $\varepsilon=\pm 1$. Thus, with any of these choice of Boltzmann weights $b$ and any integer $n\geq 2$, the $n^\text{th}$ Potts invariant of $K$ is given by
\begin{equation}
    \text{Potts}_n(K) = \left(\frac{1}{\sqrt{n}}\right)^{V+1} \left(b_-^=\right)^{w(D)}\sum_{\begin{smallmatrix}
        \text{states}
    \end{smallmatrix}} \prod_{\begin{smallmatrix}
        \text{edges }e
    \end{smallmatrix}} b(e).
\end{equation}
\end{proposition}

\begin{proof}
    We find it convenient to write $b_+^==x$ and $b_+^{\neq}=y$.
To begin, Jones' (2.3) implies that $b_-^==x^{-1}$ and $b_-^{\neq}=y^{-1}$.  
Considering Jones' (2.4) with $a= c$ and substituting gives no new information,
\begin{align*}
    b_-^=b_+^{=}+(n-1)b_-^{\neq}b_+^{\neq}&=n\\
    x^{-1}x+(n-1)y^{-1}y&=n,
\end{align*}
but considering Jones' (2.4) with $a\neq c$ and substituting gives
\begin{align*}
    b_-^=b_+^{\neq}+b_-^{\neq}b_+^=+(n-2)b_-^{\neq}b_+^{\neq}&=0\\
    x^{-1}y+y^{-1}x+(n-2)y^{-1}y&=0.
\end{align*}
(Note that setting $x=1$ here would yield Equation \eqref{E:error}.) Solving instead for $y$ gives
\begin{equation}\label{E:quadratic_formula}
\begin{split}
    y^2+yx(n-2)+x^2&=0\\
    y&=\frac{x}{2}\left(2-n+\varepsilon \sqrt{n(n-4)}\right),\\
\end{split}
\end{equation}
where $\varepsilon=\pm 1$.  We now consider Jones' (2.5) with $a=r=c$:
\begin{align*}
    b_+^{=}~ b_+^{=}~ b_-^{=}+(n-1)b_+^{\neq}~ b_+^{\neq}~ b_-^{\neq}&=\sqrt{n}\left(b_+^{=}~ b_-^{=}~ b_-^{=}\right)\\
    x~ x~ x^{-1}+(n-1)y~ y~ y^{-1}&=\sqrt{n}\left(x~ x^{-1}~ x^{-1}\right)\\
x\left(1+\frac{n-1}{2}\left(2-n+\varepsilon \sqrt{n(n-4)}\right)\right)&=\sqrt{n}x^{-1}\\
n\left(1-\frac{n-1}{2}\right)+\varepsilon \frac{n-1}{2}\sqrt{n(n-4)}&=\sqrt{n}x^{-2}\\
\frac{1}{2}\left(\sqrt{n}(3-n)+\varepsilon (n-1)\sqrt{n-4}\right)&=x^{-2},\\
\end{align*}
giving Equation \eqref{E:Potts_general_x}. The second part of Equation \eqref{E:quadratic_formula} then gives Equation \eqref{E:Potts_general_y}, and Jones' Equation (2.3) confirms the remaining two Boltzmann weights. The rest of the proposition follows from Definition \ref{D:Partition} and Notation \ref{N:IsingPotts}.
\end{proof}

We note the particular solutions when $n=2,3,4$:

\begin{corollary}
    Writing $\omega=\pm e^{\pm i\pi/8}$, the Boltzmann weights for the link-invariant Ising model are given by  $b^+_{=}=\omega,b^-_{=}=\omega^{-1},b^+_{\neq}=\omega^{-3}$, and $b^-_{\neq}=\omega^3$.
\end{corollary}

\begin{corollary}
    The Boltzmann weights for the link-invariant Potts model with $n=3$ are given by 
    \begin{equation}\label{E:Potts_3}
\left(b_+^=,b_-^=,b_+^{\neq},b_-^{\neq}\right)=\left(e^{i(1+2k)\pi/4},e^{-i(1+2k)\pi/4},e^{\frac{(3+6k\pm 8)\pi i}{12}},e^{\frac{-(3+6k\pm 8)\pi i}{12}}\right),~k=0,1,2,3.
    \end{equation}
\end{corollary}

\begin{corollary}
    The Boltzmann weights for the link-invariant Potts model with $n=4$ are given by 
    \begin{equation}\label{E:Potts_4}
\left(b_+^=,b_-^=,b_+^{\neq},b_-^{\neq}\right)=\pm (i,-i,-i,i).
    \end{equation}    
\end{corollary}

We are ready to prove our main theorem.


\begin{proof}[Proof of Theorem \ref{T:Main}]
Assume the entire setup of the theorem, including the values for $x$, $y$, and $t$, as well as the Boltzmann weights $b$ and $B$. Proposition \ref{P:Main} confirms Equation \eqref{E:Potts_with_CF}. It will suffice to confirm Equation \eqref{E:Potts}. To begin, notice that

\[\frac{b_+^{\neq}}{b_+^=}=\frac{y}{x}=t^{-1}=\frac{B_+^{\neq}}{B_+^=}
\]
and 
 therefore
 \[\frac{b_+^{\neq}}{B_+^{\neq}}=\frac{b_+^=}{B_+^{=}}=b_+^==x.\]
Similarly, $\frac{B_-^{\neq}}{b_-^{\neq}}=x^{-1}.$
Thus, for any state $S$ of 
$G$,
\begin{align*}
\prod_{\begin{smallmatrix}\text{edges }\\ e\text{ in }G \end{smallmatrix}}b(e)
=&\prod_{\begin{smallmatrix}\text{positive}\\\text{edges }\\ e\text{ in }G \end{smallmatrix}}b_+(e)\cdot \prod_{\begin{smallmatrix}\text{negative}\\\text{edges }\\ e\text{ in }G \end{smallmatrix}}b_-(e)\\
=&\left(\frac{b_+^=}{B_+^=}\right)^{E_+(G)}\left(\frac{b_-^{=}}{B_-^{=}}\right)^{E_-(G)}\prod_{\begin{smallmatrix}\text{positive}\\\text{edges }\\ e\text{ in }G \end{smallmatrix}}B_+(e)\cdot \prod_{\begin{smallmatrix}\text{negative}\\\text{edges }\\ e\text{ in }G \end{smallmatrix}}B_-(e)\\
=&x^{E_+(G)-E_-(G)}\prod_{\begin{smallmatrix}\text{edges }\\ e\text{ in }G \end{smallmatrix}}B(e).
\end{align*}
Using this correction factor of $x^{E_+-E_-}$ together with Jones' correction factors of $\left(\frac{1}{\sqrt{2}}\right)^{V+1}$ and $\left(b_-^=\right)^{w(D)}=x^{-w(D)}$ completes the proof.
\end{proof}

Finally, we prove Theorem \ref{T:Jones}

\begin{proof}[Proof of Theorem \ref{T:Jones}]
By Definition \ref{D:Jones}, it will suffice to prove, for arbitrary $n$, that $\text{Potts}_n(\bigcirc)=1$ and $t^{-1}\text{Potts}_n(K_+)-t\text{Potts}_n(K_-)=(t^{1/2}-t^{-1/2})\text{Potts}_n(K_0)$ for any three oriented link diagrams $K_+$, $K_-$, and $K_0$ that are identical outside of some disk where they appear as in Figure \ref{Fi:skeinac}.

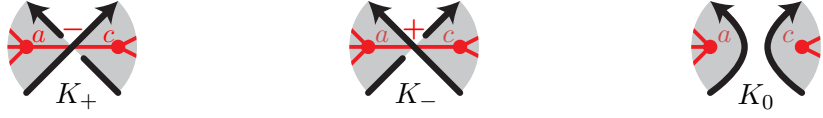
\begin{figure}
    \centering
    \labellist
    \pinlabel {$K_+$} at 55 0
    \pinlabel {$K_-$} at 322 0
    \pinlabel {$K_0$} at 589 0
    \small
    \pinlabel {{\color{red} $a$}} at 25 47
    \pinlabel {{\color{red} $c$}} at 80 47
    \pinlabel {{\color{red} $\boldsymbol{-}$}} at 52 53
    \pinlabel {{\color{red} $\boldsymbol{+}$}} at 320 53
    \pinlabel {{\color{red} $a$}} at 295 47
    \pinlabel {{\color{red} $c$}} at 347 47
    \pinlabel {{\color{red} $a$}} at 564 47
    \pinlabel {{\color{red} $c$}} at 614 47
    \endlabellist
    \includegraphics[height=.5in]{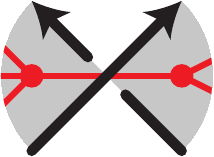}
    \hspace{1in}
    \includegraphics[height=.5in]{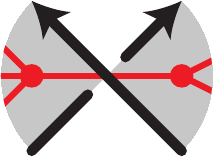}
    \hspace{1in}
    \includegraphics[height=.5in]{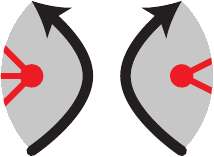}
    \caption{The local appearances of $K_+$, $K_-$, and $K_0$ in the 
    skein relation \eqref{E:skein}, shaded and labeled for the proof of Theorem \ref{T:Jones}}
    \label{Fi:skeinac}
\end{figure}

To prepare, we note that taking square roots on both sides of the equation $t+2+t^{-1}=n$ (and choosing the appropriate value for $t^{1/2}$) gives 
\begin{equation*}\label{E:Jones12}
t^{1/2}+t^{-1/2}=
-\sqrt{n},
\end{equation*}
and multiplying this equation on both sides by $-(t^{1/2}-t^{-1/2})/\sqrt{n}$ gives 
\begin{equation}\label{E:last_sub}
    \frac{t^{-1}-t}{\sqrt{n}}=t^{1/2}-t^{-1/2}.
\end{equation}
Finally, writing $G_0$ for the Tait graph of $K_0$ and denoting the associated correction factor
\[\left(\frac{1}{\sqrt{n}}\right)^{V(G_0)+1}x^{w(G_0)+E_-(G_0)-E_+(G_0)}=CF,\]
we are prepared to use Equations \eqref{E:Potts} and \eqref{E:last_sub} to complete the proof of Equation \eqref{E:third}:
\begin{align*}
t^{-1}\text{Potts}_n(K_+)-t\text{Potts}_n(K_-)
=&\frac{CF}{\sqrt{n}}\left(t^{-1}\left(\sum_{\begin{smallmatrix}
        \text{states }\sigma \\ \text{of }G_0\text{ with}\\\sigma(a)=\sigma(c) \end{smallmatrix}} \prod_{\begin{smallmatrix}
        \text{edges }\\
        e\neq\overline{ac}\\ \text{in }G_0
    \end{smallmatrix}}B(e)+\sum_{\begin{smallmatrix}
        \text{states }\sigma\\ \text{of }G_0\text{ with}\\\sigma(a)\neq \sigma(c) \end{smallmatrix}} t\prod_{\begin{smallmatrix}
        \text{edges }\\
        e\neq\overline{ac}\\ \text{in }G_0
    \end{smallmatrix}}B(e)\right)\right.\\
    &\hspace{.35in}-t\left.
    \left(\sum_{\begin{smallmatrix}
        \text{states }\sigma \\ \text{of }G_0\text{ with}\\\sigma(a)=\sigma(c) \end{smallmatrix}} \prod_{\begin{smallmatrix}
        \text{edges }\\
        e\neq\overline{ac}\\ \text{in }G_0
    \end{smallmatrix}}B(e)+\sum_{\begin{smallmatrix}
        \text{states }\sigma\\ \text{of }G_0\text{ with}\\\sigma(a)\neq \sigma(c) \end{smallmatrix}} t^{-1}\prod_{\begin{smallmatrix}
        \text{edges }\\
        e\neq\overline{ac}\\ \text{in }G_0
    \end{smallmatrix}}B(e)\right)\right)
    \\
    =&\frac{CF}{\sqrt{n}}\left(t^{-1}-t\right)\sum_{\begin{smallmatrix}
        \text{states }\sigma\\  \text{of }G_0\text{ with}\\\sigma(a)=\sigma(c) \end{smallmatrix}} \prod_{\begin{smallmatrix}
        \text{edges }\\
        e\neq\overline{ac}\\ \text{in }G_0
    \end{smallmatrix}}B(e)\\
    =&CF(t^{1/2}-t^{-1/2})\sum_{\begin{smallmatrix}
        \text{states }\sigma\\ \text{of }G_0\text{ with}\\\sigma(a)=\sigma(c) \end{smallmatrix}} \prod_{\begin{smallmatrix}
        \text{edges }\\
        e\neq\overline{ac}\\ \text{in }G_0
    \end{smallmatrix}}B(e)\\
    =&(t^{1/2}-t^{-1/2})\text{Potts}_n(K_0).
\end{align*}
The last comment in the theorem---the connection between the Ising and arf invariants of an oriented link---follows immediately from Equation \eqref{E:third} and Theorem 10.6 of \cite{lickorish}.
\end{proof}

\begin{table}[h!]
\begin{center}
\renewcommand{\arraystretch}{1.5}
\begin{tabular}{|| c || c | c | c | c ||}
 \hline
\raisebox{18pt}{move} & \includegraphics[height=48pt]{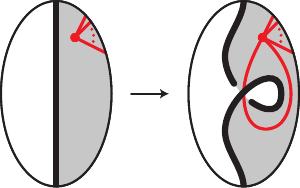} & \includegraphics[height=48pt]{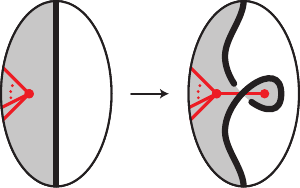} & \includegraphics[height=48pt]{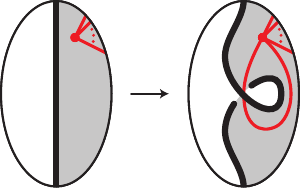} & \includegraphics[height=48pt]{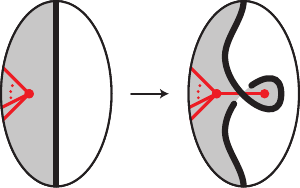}\\
\hline
$\Delta V$ & 0 & 1 & 0 & 1\\
$\Delta E_+$ & 1 & 0 & 0 & 1\\
$\Delta E_-$ & 0 & 1 & 1 & 0\\
$\Delta w(D)$ & 1 & 1 & $-1$ & $-1$ \\ \hline
\renewcommand{\arraystretch}{1}
$\begin{matrix} \text{correction}\\ \text{factor needed}\end{matrix}$& 1 & $\frac{1}{1-i}=\frac{1}{\sqrt{2}}\omega^2$ & 1 & $\frac{1}{1+i}=\frac{1}{\sqrt{2}}\omega^{-2}$ \\
\hline\hline
\end{tabular}
\caption{Reidemeister 1 moves gave part of the system of equations that we used to determine the correction factor in Equation \eqref{E:Potts}}
\label{T:R1}
\end{center}
\end{table}

\begin{table}[h!]
\begin{center}
\renewcommand{\arraystretch}{1.5}
\begin{tabular}{|| c || c | c | c | c ||} 
 \hline
\raisebox{18pt}{move} & \includegraphics[height=48pt]{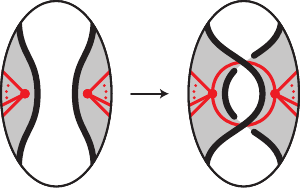} & \includegraphics[height=48pt]{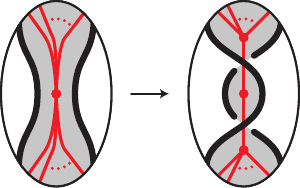} & \includegraphics[height=48pt]{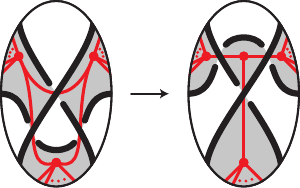} & \includegraphics[height=48pt]{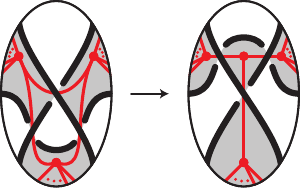}\\
\hline
$\Delta V$ & 0 & 2 & 1 & 1\\
$\Delta E_+$ & 1 & 1& 1 & $-1$\\
$\Delta E_-$ & 1 & 1& $-1$ &$1$\\
$\Delta w(D)$ & 0 & 0 &0& 0\\ \hline
\renewcommand{\arraystretch}{1}
$\begin{matrix} \text{correction}\\ \text{factor needed}\end{matrix}$& 1 & $\frac{1}{2} $& $\frac{1}{1+i}=\frac{1}{\sqrt{2}}\omega^{-2}$ & $\frac{1}{1-i}=\frac{1}{\sqrt{2}}\omega^2$\\
\hline\hline
\end{tabular}
\caption{Reidemeister 2 and 3 moves gave the rest of the system of equations that we used to determine the correction factor in Equation \eqref{E:Potts}}
\label{T:R23}
\end{center}
\end{table}

\begin{remark}
    The proofs we have presented to justify Equations \eqref{E:Potts}, \eqref{E:Potts_with_CF}, and \eqref{E:third} differ substantially from the manner in which we derived these formulas.  Firstly, the vast majority of our time and attention was devoted to the case of the Ising model; once everything was carefully worked out for $n=2$, the general formulas (for the Potts models) naturally fell into place. Secondly, our exploration of the Ising model began with a consideration of the Boltzmann weights of \[B_+^{=}=1=B_-^{=},~B_+^{\neq}=i,\text{ and }B_-^{\neq}=-i,\] which Jones proposes in Example 2.17 of \cite{jones_pacific} (and Adams repeats in \cite{knotbook}).  Although we noticed that the resulting partition function would not always remain invariant under Reidemeister moves, we wondered whether it might be possible to remedy this non-invariance by means of a correction factor.  We considered each Reidemeister move and the type of correction factor it would need. As shown in Tables \ref{T:R1} and \ref{T:R23}, this yielded a system of equations for how the correction factor must depend on the writhe of the diagram, $w(D)$, and various features of the Tait graph---the number of vertices, $V$, and the numbers of positive and negative edges, $E_+$ and $E_-$. Solving that system of equations yielded the correction factor that appears in Equation \eqref{E:Potts} (setting $n=2$), namely $
    (
    \frac{1}{\sqrt{2}}
    )^{V+1}  \omega^{w(D)+E_--E_+}$, where $\omega=\pm e^{ i\pi/8}$.  
\end{remark}

\end{document}